\documentclass[11pt]{article}
\usepackage[utf8]{inputenc}
\usepackage{amsfonts}
\usepackage{geometry}
\usepackage{fullpage}
\usepackage{amssymb}
\usepackage{amsthm}
\usepackage{graphicx}
\usepackage[all,2cell]{xy}
\usepackage{listings}
\usepackage{url}
\usepackage{sectsty}
\usepackage{graphicx}
\usepackage{geometry}
\usepackage{amsmath}
\sectionfont{\Large}
\subsectionfont{\Large}
\subsubsectionfont{\Large}
\allowdisplaybreaks
\newtheorem{theorem}{Theorem}[section]
\newtheorem{corollary}[theorem]{Corollary}
\newtheorem{lemma}[theorem]{Lemma}

\newtheorem{proposition}[theorem]{Proposition}

\newtheorem{conjecture}[theorem]{Conjecture}

\theoremstyle{definition}
\newtheorem{example}[theorem]{Example}

\newtheorem{remark}[theorem]{Remark}

\title{Invariant theory of finite general linear groups modulo Frobenius powers}
\author{Pallav Goyal}
\date{}

\begin{document}

\maketitle

\begin{abstract}
We prove some cases of a conjecture of Lewis, Reiner and Stanton regarding Hilbert series corresponding to the action of $Gl_n(\mathbb{F}_q)$ on a polynomial ring modulo Frobenius powers. We also give a few conjectures about the invariant ring for certain cases that we don't prove completely.
\end{abstract}

\section{Introduction} \label{sec:intro}

The subject of this paper is a conjecture inspired by the following celebrated result by L.E. Dickson \cite{LED}:

\begin{theorem} [Dickson]
When $G=Gl_n(\mathbb{F}_q)$ acts via invertible linear substitution of variables on the polynomial ring $S = \mathbb{F}_q[x_1,x_2,\cdots,x_n]$, the $G$-invariant ring is given as:
$$S^G = \mathbb{F}_q[D_{n,0},D_{n,1},\cdots,D_{n,n-1}].$$
\end{theorem}

Here, the $D_{n,i}$'s are the Dickson polynomials defined via the identity: 
$$\prod (t + l(\textbf{x})) = \sum_{i=0}^n D_{n,i}t^{q^i}$$
where the product is taken over all linear functionals $l(\textbf{x})$ over $S$. Building on this, we wish to consider the action of $G$ on the ring $Q = S/\mathfrak{m}^{[q^m]}$ where $\mathfrak{m}^{[q^m]} = (x_1^{q^m},x_2^{q^m},\cdots,x_n^{q^m})$. As $\mathfrak{m}^{[q^m]}$ remains invariant under the action of $G$, the action of $G$ on $Q$ is well defined. We wish to describe the structure of the invariant ring $Q^G$ and work towards proving the following conjecture by J. Lewis, V. Reiner and D. Stanton \cite{LRS}:

\begin{conjecture} \label{conj:prob}
The Hilbert series for the above action of $Gl_n(\mathbb{F}_q)$ on $Q$ is given by:
$$Hilb(Q^G,t)=\sum\limits_{k=0}^{\min(n,m)}t^{(n-k)(q^m-q^k)}{m \brack k}_{q,t}$$
where ${m \brack k}_{q,t} = \prod\limits_{i=0}^{k-1} \frac{1-t^{q^m-q^i}}{1-t^{q^k-q^i}}$.
\end{conjecture}

The layout of the paper is as follows. Section~\ref{sec:m=1} deals with the fairly easy case of the problem when $m=1$. Although this case has already been dealt with in \cite{LRS}, it will set the stage for the following sections to flow more naturally. Section~\ref{sec:m=2} tackles the case $m=2$ in $2$ stages, first assuming that $n=2$, and then, generalising for an arbitrary $n$. Section~\ref{sec:k=1} uses the ideas developed in Section~\ref{sec:m=2} to talk about the `$k=1$' case of the Hilbert series. Section~\ref{sec:SO} is about the action of Steenrod operations on the invariant polynomials when $m=2$. Section~\ref{sec:n=k=m-1} develops certain ideas that approach a solution for the case $n=k=m-1$. Section~\ref{sec:q=2} records certain simplifications that arise when we take $\mathbb{F}_2$ as the underlying field. Finally, Section~\ref{sec:para} deals with the parabolic generalization of the conjecture with focus on the case $m=2$.

\subsection{Preliminaries}

Before starting off, we state a few results which will be used multiple times throughout the paper.

\begin{lemma} \label{lemma:intro1}
A polynomial is invariant under the action of $Gl_n(\mathbb{F}_q)$ if it is invariant under diagonal matrices, permutations of the variables, and the substitution $x_1 \to x_1 + x_2$.
\end{lemma}

\begin{proof}
First, we note that the group $Gl_n(\mathbb{F}_q)$ is generated by diagonal matrices, permutation matrices and elementary matrices. Hence, by symmetry, checking that a given polynomial is invariant amounts to checking its invariance under the former two and the substitution $x_1 \rightarrow x_1+x_2$.
\end{proof}

\begin{lemma} \label{lemma:intro}
$Q$ has a basis of monomials of the form $x_1^{i_1} \cdots x_n^{i_n}$ where $0 \leq i_j \leq q^n-1$ as an $\mathbb{F}_q$-vector space. Furthermore, if $y \in Q^G$, then $y$ is a linear combination of monomials where each $i_j$ is divisible by $q-1$.
\end{lemma}

\begin{proof}
The first claim of the lemma is rather obvious. For the second claim, we note that for $y$ to be invariant, it must be invariant under the action of the diagonal matrices.
\end{proof}

Finally, we state a well known result from number theory. For proof, see \cite[Page 126]{LT}.

\begin{theorem} [Lucas' Theorem]
Given positive integers $m$ and $n$, if $n=n_0+n_1p+n_2p^2+\cdots+n_dp^d$ and $m=m_0+m_1p+m_2p^2+\cdots+m_dp^d$ are their respective base $p$ expansions, then:
$${n \choose m} \equiv {n_0 \choose m_0}{n_1 \choose m_1} \cdots {n_d \choose m_d} \pmod{p}.$$
\end{theorem}

\section{The case $m=1$} \label{sec:m=1}

We have $Q=\mathbb{F}_q[x_1,x_2,\cdots ,x_n]/\mathfrak{m}^{[q]}$ where $\mathfrak{m}^{[q]}=(x_1^q,x_2^q,\cdots ,x_n^q)$. Let $y$ be an element of the invariant ring $Q^G$. Then, by Lemma~\ref{lemma:intro}, $y$ must be a linear combination of elements of the form $x_1^{i_1}x_2^{i_2}\cdots x_n^{i_n}$ where $i_j=0$ or $q-1$ for all $j$. Now, let $bx_1^{j_1}x_2^{j_2}\cdots x_n^{j_n}$ be a monomial occurring in $y$ where $b \in \mathbb{F}_q \setminus \{0\}$. 

I claim that $j_1=j_2=\cdots =j_n$. Suppose not. Without loss of generality, let $j_1=j_2=\cdots=j_r=q-1$ and $j_{r+1}=\cdots =j_n=0$ where $0<r<n$. Hence, the monomial is $bx_1^{q-1}x_2^{q-1}\cdots x_r^{q-1}$. If we perform the substitution $x_r \rightarrow x_r + x_{r+1}$ in $y$, the coefficient of $x_1^{q-1}x_2^{q-1}\cdots x_r^{q-2}x_{r+1}$ in $y$ becomes non-zero. But this isn't possible as $y$ does not contain any such monomial except possibly when $q=2$. In that particular case, the coefficient of $x_1^{q-1}x_2^{q-1}\cdots x_r^{q-2}x_{r+1}=x_1x_2\cdots x_{r-1}x_{r+1}$ in $y$ either changes from $1$ to $0$ or from $0$ to $1$, which isn't allowed in the invariant $y$. Hence, we have a contradiction.

Therefore, $y$ is of the form $c_0 + c_1x_1^{q-1}x_2^{q-1}\cdots x_n^{q-1}$. It is trivial to see that such a $y$ is invariant under the action of $G$. Hence, the Hilbert series is $1+t^{n(q-1)}$, which agrees with what has been conjectured.

\section{The case $m=2$} \label{sec:m=2}

We first solve the problem for the case when $n=2$, and then show how the solution can be generalised for arbitrary $n$.

\subsection{$m=2, n=2$} \label{subsec:n=2}

We have $Q=\mathbb{F}_q[x_1,x_2]/\mathfrak{m}^{[q^2]}$ where $\mathfrak{m}^{[q^2]}=(x_1^{q^2},x_2^{q^2})$. By Lemma~\ref{lemma:intro}, we see that any $y \in Q^G$ is a linear combination of elements of the form $x_1^{i_1}x_2^{i_2}$ where $0 \leq i_1, i_2 \leq q^2-1$ and $(q-1)|i_1, (q-1)|i_2$. We'll now prove the following theorem which is in accord with the conjecture at hand.

\begin{theorem} \label{theo:m=2,n=2}
The Hilbert series for $Q^G$ is
$$1 + t^{q^2-q}(1+ t^{q-1} + t^{2(q-1)} + \cdots + t^{q^2-q}) + t^{2(q^2-1)}.$$ 
A basis of the invariant ring as an $\mathbb{F}_q$-vector space is given as 
$$1, x_1^{k(q-1)}x_2^{k(q-1)}\frac{x_1^{(q-k+1)(q-1)} - x_2^{(q-k+1)(q-1)}}{x_1^{q-1} - x_2^{q-1}} \text{and }  x_1^{q^2-1}x_2^{q^2-1}$$
for $0\leq k \leq q$.
\end{theorem}

In order to prove this, we'll show that there exist unique invariant polynomials in $Q$ of degrees $q^2-q + c(q-1)$ and $2(q^2-1)$ where $0 \leq c \leq q$, and that there do not exist any other invariant polynomials (up to scalar multiples). Without loss of generality, let us assume that $y$ is a homogeneous $G$-invariant polynomial. Then, the degree $d$ of $y$ must be some multiple of $q-1$ between $1$ and $2(q^2-1)$. Let $d=k(q-1)$, where $0 \leq k \leq 2(q+1)$.

\begin{lemma}
If $k<q$ or $k=2q+1$, we don't have any invariant polynomials $y$ of degree $k(q-1)$.
\end{lemma}

\begin{proof}
First, consider the case $q=2$. If $k<q$, we have $k=1$, and so, $y$ is of the form $c_0x_1+c_1x_2$. If $k=2q+1=5$, $y$ is of the form $c_0x_1^3x_2^2+c_1x_1^2x_2^3$. Neither of these can be invariant under the substitutions $x_1\rightarrow x_1+x_2$ and $x_2 \rightarrow x_1+x_2$ unless $c_0=c_1=0.$

Now, suppose $q>2$ and $k<q$. Let $k=p^tr$ where $p \nmid r$. A general expression of $y$ is the form:
$$y=c_0x_1^{k(q-1)} + c_1x_1^{(k-1)(q-1)}x_2^{q-1} +\cdots+c_kx_2^{k(q-1)}.$$ 

Consider the substitution $x_1 \rightarrow x_1 + x_2$. Then, the coefficient of $x_1^{k(q-1) - p^t}x_2^{p^t}$ is equal to ${k(q-1) \choose p^t} = {p^tr(q-1) \choose p^t} \equiv r(q-1) \not\equiv 0 $ mod $p$ (where the second equality follows by Lucas's Theorem). Hence, the coefficient becomes non-zero after the substitution if $c_0 \neq 0$. So, let $c_0=0$. By exactly the same argument, we get that $c_j=0$ for all $j$. Hence, $y=0$ contradicting the fact that its degree is non-zero. Thus, such a $y$ can not be invariant. Hence, $k \geq q$.

Next, let $k=2q+1$. In this case, $y=c(x_1^{q^2-1}x_2^{q^2-q} + x_1^{q^2-q}x_2^{q^2-1})$. But here, the substitution $x_1 \rightarrow x_1 + x_2$ gives a non-zero coefficient to $x_1^{q^2-2}x_2^{q^2-q+1}$ unless $c=0$. Thus, we do not have an invariant of degree $(2q+1)(q-1)$.
\end{proof}

The cases left to investigate are $q \leq k \leq 2q$ and $k=2(q+1)$. For $k=2(q+1)$, there's a unique choice for $y$, i.e. $y = cx_1^{q^2-1}x_2^{q^2-1}$. It is easy to see that this $y$ remains invariant under all the possible transformations from $G$. 

\begin{lemma}
For $q\leq k\leq 2q$, we have unique (up to scalar multiples) invariant polynomials of degree $k(q-1)$.
\end{lemma}

\begin{proof}
Fix $k$ and let $k'=k-q$. Then, a general expression of an invariant polynomial $y$ of degree $k$ is of the form:
$$y=c_0x_1^{q^2-1}x_2^{(k'-1)(q-1)} + c_1x_1^{q^2-q}x_2^{k'(q-1)} + \cdots + c_{q-k'+2}x_1^{(k'-1)(q-1)}x_2^{q^2-1}.$$ 

I claim that $c_0=0$ and $c_1 \neq 0$. If $c_0 \neq 0$, then on the substitution $x_1 \rightarrow x_1+x_2$, the coefficient of $x_1^{q^2-2}x_2^{(k'-1)(q-1)+1}$ becomes $c_0(q^2-1)\neq 0$, which isn't allowed. Thus, $c_0=0$. Next suppose $c_1=0$. Choose the smallest $i$ such that $c_i \neq 0$. In the above polynomial, $c_i$ is the coefficient of $x_1^{(q-i+1)(q-1)}x_2^{(k'+i-1)(q-1)}$. Let $q-i+1=p^tr$ where $p \nmid r$. As $i \neq 1$, $p^t < q$. Then, the same substitution gives a non-zero coefficient to $x_1^{(q-i+1)(q-1) - p^t}x_2^{(k'+i-1)(q-1) + p^t}$ which is again impossible. Hence, $c_1 \neq 0$. Without loss of generality, $c_1 = 1$. I claim that
$$y=y_{k'} := x_1^{q^2-q}x_2^{k'(q-1)}+ x_1^{(q-1)^2}x_2^{(k'+1)(q-1)}+\cdots+x_1^{k'(q-1)}x_2^{q^2-q}.$$

That $y_{k'}$ is invariant, will be proven in a more general setting in Proposition~\ref{prop:k=1}. So, we only need to check that this choice of $y$ is forced, and thus, unique. But, this is easy to see. Having fixed $c_1=1$, after the substitution $x_1 \rightarrow x_1+x_2$, the coefficient of $x_1^{(q-i+1)(q-1) - p^t}x_2^{(k'+i-1)(q-1) + p^t}$ should be equal to $0$ (except in the case when $q=2$ and $k=1$ when this coefficient should be $1$), and using this, the values of the other $c_i$'s gets uniquely determined.
\end{proof}

We illustrate the last step of the proof with an example. 

\begin{example}
Consider the case $q=4$ and $k'=0$. Then, a general expression for $y$ is of the form $c_1x_1^{12} + c_2x_1^9x_2^3 + c_3x_1^6x_2^6 + c_4x_1^3x_2^9 + c_5x_2^{12}$. As above, we'll assume $c_1=1$ and perform the substitution $x_1 \rightarrow x_1 + x_2$. Then, the coefficient of $x_1^8x_2^4$ becomes $c_1{12 \choose 4} + c_2{9 \choose 1}$ which should be $0$. This determines $c_2$ as ${9 \choose 1} \neq 0$. Next, we try to find $c_3$. As $2|6$ and $4 \nmid 6$, we consider the coefficient of $x_1^{6-2}x_2^{6+2} = x_1^8x_2^4$. This is equal to $c_1{12 \choose 8} + c_2{9 \choose 5} + c_3{6 \choose 2}$ which must be $0$. Again, as ${6 \choose 2} \neq 0$, $c_3$ gets uniquely determined. Proceeding similarly, considering the coefficient of $x_1^2x_2^{10}$, we get the value of $c_4$, and then, $c_5 = c_1$ by symmetry. Hence, $c_1$ determines all the other coefficients uniquely.
\end{example}

\begin{remark}
Up till now, we've only described the invariant ring $Q^G$ in terms of a linear basis as an $\mathbb{F}_q$-vector space. Some multiplicative properties can be easily observed. Recall that $y_{k'} = x_1^{q^2-q}x_2^{k'(q-1)}+ x_1^{(q-1)^2}x_2^{(k'+1)(q-1)}+\cdots+x_1^{k'(q-1)}x_2^{q^2-q}$ for $0\leq k' \leq q$. It is a matter of simple computations to check that $y_0^2 = y_q$, $y_0y_2=-x_1^{q^2-1}x_2^{q^2-1}$ and $y_2^2=0$. All the products $y_iy_j$, apart from these 3, are zero because of degree arguments. For example, $y_0y_1$ has degree $2q^2-q-1$, but we know that there aren't any invariant polynomials of this degree, and so, $y_0y_1=0$.
\end{remark}

With this, we have now proven Theorem~\ref{theo:m=2,n=2} and have a complete description of the structure of the invariant ring.

\subsection{$m=2$, arbitrary $n$} \label{subsec:n=arbit}

Now, we work towards solving the problem for a general $n$ when $m=2$. We have that $Q=\mathbb{F}_q[x_1,x_2, \cdots, x_n]/\mathfrak{m}^{[q^2]}$ where $\mathfrak{m}^{[q^2]}=(x_1^{q^2},x_2^{q^2},\cdots,x_n^{q^2})$. The aim of this section will be to generalise the ideas from the case $n=2$. We define the polynomials:
$$z_n := \prod\limits_{i=1}^n x_i^{q^2-1}$$
and
$$y_{n,k} : = \sum\limits_{\substack{i_1+i_2+\cdots+i_n=(n-1)q+k \\ 0\leq i_1,i_2,\cdots,i_n \leq q }}\bigg(\prod\limits_{j=1}^n x_j^{i_j(q-1)}\bigg)$$
for $0 \leq k \leq q$. In fact, $y_{n,k}$ can be defined for $k>q$ using the same definition. But in those cases, the bounds on the $i_j$'s would imply that the sum is an empty sum, and so, $y_{n,k}=0$ for $k>q$. 

\begin{lemma} \label{lemma:invar}
$z_n$ and $y_{n,k}$ remain invariant under the action of $G$.
\end{lemma}

\begin{proof}
The proof is by induction. The case when $n=1$ is easily verified and the case $n=2$ has already been solved and acts as base case of the induction.

$z_n$ can be trivially seen to be invariant under the action of $Gl_n(\mathbb{F}_q)$ on $Q$. Now, fixing $k$, we consider $y_{n,k}$. As this expression is symmetric in the $x_i$'s, it remains invariant under the action of the permutation matrices. Also, as all the exponents are divisible by $q-1$, diagonal matrices too act trivially on $y_{n,k}$. Therefore, by Lemma~\ref{lemma:intro1}, we only need to check the action of the substitution $x_1 \rightarrow x_1+x_2$ on $y_{n,k}$. Here, we make the observation that $y_{n,k}$ can be written as:
\begin{align*}
y_{n,k} &= \sum\limits_{\substack{i_1+i_2+\cdots+i_n=(n-1)q+k \\ 0\leq i_1,i_2,\cdots,i_n \leq q }}\bigg(\prod\limits_{j=1}^n x_j^{i_j(q-1)}\bigg)
\\&= \sum\limits_{i_n=0}^{q}\Bigg(\sum\limits_{\substack{i_1+i_2+\cdots+i_{n-1}=(n-1)q+k-i_n \\ 0\leq i_1,i_2,\cdots,i_{n-1} \leq q }}\bigg(\prod\limits_{j=1}^{n-1} x_j^{i_j(q-1)}\bigg)\Bigg)x_n^{i_n(q-1)}
\\&= \sum\limits_{i_n=0}^{q} y_{n-1,q+k-i_n}x_n^{i_n(q-1)}
\\&= \sum\limits_{i=k}^{q} y_{n-1,q+k-i}x_n^{i(q-1)}
\end{align*}
where, in the last step, we have discarded some of the initial terms of the sum, because $y_{n-1,q+k-i_n}$ is non-zero only when $ q+k-i_n \leq q$. When we apply the transformation $x_1 \rightarrow x_1 + x_2$ to the final expression obtained above, $y_{n-1,q+k-i}$ remains invariant by the induction hypothesis and $x_n$ remains invariant as $n>2$. Hence, $y_{n,k}$ remains invariant under the action of $G$. 
\end{proof}

Now, we state the main result of this section:

\begin{theorem}
The Hilbert series for $Q^G$ is given as:
$$Hilb(Q^G,t)=1+t^{(n-1)(q^2-q)}(1+t^{q-1}+t^{2(q-1)}+\cdots+t^{q^2-q}) + t^{n(q^2-1)}$$
when $n \geq 2$. For $n=1$, the Hilbert series is $\frac{1-t^{q^2+q-2}}{1-t^{q-1}}$.
A basis for the invariant ring $Q^G$ as an $\mathbb{F}_q$-vector space is given as $1$, $y_{n,k}$ and $z_n$ for $0 \leq k \leq q$.
\end{theorem}

As the case when $n=1$ is trivially true, we'll assume $n\geq2$ here onwards. By Lemma~\ref{lemma:invar} we only need to check that these are the only invariants in $Q$ up to scalar multiples. The proof will be done in $2$ steps. Firstly, we'll try to find the possible degrees for an invariant polynomial. Next, for those degrees, we'll prove that there exist unique invariant polynomials up to scalar multiples. As we have already proven the existence of at least one such polynomial for each degree, we'll be done. 

\textbf{Step 1:} Let $y \in Q^G$ be a homogeneous polynomial of degree $r$. Express $y$ as:
\begin{equation} 
y = c_0 + c_1x_n^{q-1} + c_2x_n^{2(q-1)} + \cdots + c_kx_n^{k(q-1)} \label{eq:1}
\end{equation}
for some $k$ where $c_i \in \mathbb{F}_q[x_1,x_2,\cdots,x_{n-1}]/(x_1^{q^2},x_2^{q^2},\cdots,x_{n-1}^{q^2})$ for all $i$ and $c_k \neq 0$. If $y$ is to be invariant under the action of $G$, each of the $c_i$'s must also be invariants. By the induction hypothesis, any such non-scalar $c_i$ must have degree greater than or equal to $(n-2)(q^2-q)$. Therefore, degree of $y$ is greater than or equal to $(n-2)(q^2-q)$. We claim that $r=n(q^2-1)$ or $(n-1)(q^2-q) \leq r \leq n(q^2-q)$. The following lemmas prove this by showing that $r$ can not take any other values.

\begin{lemma}
For an invariant polynomial $y$, the degree $r$ can not be less that $(n-1)(q^2-q).$
\end{lemma}

\begin{proof}
As we have already shown that $r \geq (n-2)(q^2-q)$, suppose $(n-2)(q^2-q) \leq r < (n-1)(q^2-q)$. We can assume that none of the $c_i$'s is a scalar multiple of $z_{n-1}$, because due to symmetry, we would need to multiply it by $x_n^{q^2-1}$ which would exceed the assumed bound on degrees. Also, we must have that at least one of the $c_i$'s is non-scalar, otherwise, $y$ won't be symmetric. In any such $c_i$, the maximum degree of any $x_j$ for $j \leq n-1$ is $q^2-q$. So, the maximum degree of $x_n$ in $y$, which is $k(q-1)$, should also be equal to $q^2-q$, and thus, $k=q$. 

Next, $c_k=c_q$ can not be scalar, because that would imply that the degree of $y$ is equal to $q^2 -q < (n-2)(q^2-q)$ except when $n=3$ (in which case, we would need to have an invariant of degree $q^2-q$ which is not possible), and thus, give a contradiction. So, $c_q$ is not scalar, and so, the degree of $c_qx_n^{q^2-q}$ is greater than or equal to $(n-2)(q^2-q) + (q^2-q) = (n-1)(q^2-q)$ which exceeds the bound on the degree of $y$, and so, such a $y$ can't exist. 
\end{proof}

\begin{lemma}
For an invariant polynomial $y$, if $r$ is the degree, we can't have $n(q^2-q) < r < n(q^2-1)$.
\end{lemma}

\begin{proof}
Let us assume the contrary. Again writing $y$ in the form as in Equation~\eqref{eq:1}, we get the following similar observations: $k=q$, at least one of the $c_i$'s must be non-scalar and none of the $c_i$'s is a scalar multiple of $z_{n-1}$. This implies that deg($y$) $=$ $r=$ degree of $c_qx_n^{q(q-1)} \leq (n-1)(q^2-q) + (q^2-q) = n(q^2-q)$ which is again a contradiction. Hence, an invariant polynomial can't have its degree in the given range.
\end{proof}

\textbf{Step 2:} Step 1 makes it clear that the degree of $y$ is either $0$, $n(q^2-1)$ or a multiple of $q-1$ between $(n-1)(q^2-q)$ and $n(q^2-q)$ (both included). Now, we verify the uniqueness. The fact is trivially true for an invariant of degree $0$ or $n(q^2-1)$. 

Now, suppose we have an invariant $y$ of degree $((n-1)q+k)(q-1)$ for some $k$ such that $0\leq k \leq q$. First, we express $y$ as in Equation~\eqref{eq:1}. In this sum, each $c_i$ must be some scalar multiple of $y_{n-1,q+k-i}$. Without loss of generality, let $c_q=y_{n-1,k}$. Now, each of the $c_i$'s contain some monomial (say $M_i$) having $x_1^{q^2-q}$ as a factor. As $y$ is assumed to be invariant, applying the transformation $x_1 \rightarrow x_n$, $x_n \rightarrow x_1$, $M_i$ transforms into a monomial occurring in $c_qx_n^{q^2-q} = y_{n-1,k}x_n^{q^2-q}$. All such monomials have scalar coefficient equal to $1$ by induction. Hence, the scalar coefficient of $M_i$ must be $1$ too. This implies that for each $i$, $c_i=y_{n-1,q+k-i}$, and thus, given its degree, $y$ is determined uniquely up to scalar multiples. This completes the proof of the theorem.

\section{The case $k=1$} \label{sec:k=1}

After having dealt with the case $m=2$, it becomes easier to make some predictions for the invariant ring in the general case. Recalling the conjectured Hilbert series for the invariant ring:
$$Hilb(Q^G,t)=\sum\limits_{k=0}^{\min(n,m)}t^{(n-k)(q^m-q^k)}{m \brack k}_{q,t}.$$

When $k=1$, the summand in the above expression is equal to $t^{(n-1)(q^m-q)}{m \brack 1}$. Inspired by our discussion above, we claim that we can find invariant polynomials whose degrees correspond to this `$k=1$' term of the Hilbert series. For simplicity, first we assume that $n=2$.

\begin{proposition} \label{prop:k=1}
Taking $Q = \mathbb{F}_q[x_1,x_2]/(x_1^{q^m},x_2^{q^m})$, the polynomial $y_{k'}$ is $G$-invariant for $0 \leq k' \leq \frac{q^m-q}{q-1}$, where:
$$y_{k'}:= x_1^{q^m-q}x_2^{k'(q-1)}+x_1^{q^m-2q+1}x_2^{(k'+1)(q-1)}+\cdots +x_1^{(k'+1)(q-1)}x_2^{q^m-2q+1}+x_1^{k'(q-1)}x_2^{q^m-q}.$$
\end{proposition}

\begin{proof}

In order to prove this, we need to check that each $y_{k'}$ remains invariant under the transformation $x_1 \rightarrow x_1 + x_2$. First, let $k'=0$. Then, after the above substitution, $y_0$ becomes:
\begin{equation}
(x_1+x_2)^{q^m-q}+(x_1+x_2)^{q^m-2q+1}x_2^{q-1}+\cdots +(x_1+x_2)^{q-1}x_2^{q^m-2q+1}+x_2^{q^m-q}.
\end{equation}

The coefficient of $x_1^tx_2^{q^m-q-t}$ in the substituted expression is ${q^m-q \choose t} + {q^m-2q+1 \choose t} + \cdots + {q-1 \choose t} + {0 \choose t}$ where those binomial coefficients which do not make sense are assumed to be zero. This sum is also the coefficient of $x^t$ in:
\begin{align*}
(1+x)^{q^m-q}+(1+x)^{q^m-2q+1}+\cdots+(1+x)^{q-1}+1&=\frac{(1+x)^{(q^m-1)}-1}{(1+x)^{q-1}-1}
\\&=\frac{\frac{(1+x)^{q^m}}{1+x}-1}{\frac{(1+x)^q}{(1+x)}-1}
\\&=\frac{\frac{1+x^{q^m}}{1+x}-1}{\frac{1+x^q}{1+x}-1}
\\&=\frac{x^{q^m}-x}{x^q-x}
\\&=\frac{x^{q^m-1}-1}{x^{q-1}-1}
\\&=1+x^{q-1}+x^{2(q-1)} + \cdots + x^{q^m-q}.
\end{align*}
And so, the coefficient of $x_1^tx_2^{q^m-q-t}$ in the substituted expression is $1$ if and only if $q-1|t$ and is $0$ otherwise. Thus, the substituted expression is equal to $y_0$, and thus, $y_0$ remains invariant under the action of $G$.

\begin{remark} \label{rem:DI}
Up till now, we haven't used the fact that $x_1^{q^m}=x_2^{q^m}=0$. Hence, for each $m$, $y_0$ is invariant in $S$ itself.
\end{remark}

For any other $k'$, the transformation $x_1 \rightarrow x_1 + x_2$ turns $y_{k'}$ to:
\begin{equation}
(x_1+x_2)^{q^m-q}x_2^{k'(q-1)} + (x_1+x_2)^{q^m-2q+1}x_2^{(k'+1)(q-1)}+\cdots +(x_1+x_2)^{k'(q-1)}x_2^{q^m-q}.
\end{equation}

Here, we need to check the coefficient of $x_1^tx_2^{q^m+(k'-1)(q-1)-1-t}$ for $(k'-1)(q-1) \leq t \leq q^m-q$. This is because, for smaller $t$, the exponent of $x_2$ exceeds $q^m-1$, and so, the term becomes zero. For $t > (k'-1)(q-1)$, the coefficient is equal to ${q^m-q \choose t} + {(q^m-2q+1 \choose t} + \cdots + {k'(q-1) \choose t}$ which is the coefficient of $x^t$ in
\begin{equation}
(1+x)^{q^m-q}+(1+x)^{q^m-2q+1}+\cdots+(1+x)^{k'(q-1)}.
\end{equation}

In fact, we can continue the above sum up to $1$ as the addition of these terms won't contribute to a coefficient of $x_t$. Hence, working similarly as when $k'=0$, we get the desired coefficients. For $t=(k'-1)(q-1)$, we get coefficient equal to $1$ because of the addition of the term $(1+x)^{(k'-1)(q-1)}$ to the above sum. Subtracting this, we get its coefficient to be zero, and so, we conclude that $y_k'$ remains invariant under the given transformation.
\end{proof}

This proposition also provides us the proof of the invariance of the polynomials we talked about in Section~\ref{subsec:n=2}.

Now, for a general $n$, mimicking the strategy followed by us in Section~\ref{subsec:n=arbit}, we have an easy corollary to the above proposition.

\begin{corollary}
$$a_{m,n,k'} : = \sum\limits_{\substack{i_1+i_2+\cdots+i_n=(n-1)\frac{q^m-q}{q-1}+k' \\ 0\leq i_1,i_2,\cdots,i_n \leq \frac{q^m-q}{q-1} }}\bigg(\prod\limits_{j=1}^n x_j^{i_j(q-1)}\bigg)$$ 

for $0 \leq k' \leq \frac{q^m-q}{q-1}$ are invariant polynomials in the ring $Q$ for any $n,m$.
\end{corollary}

The proof of this corollary uses exactly the same idea as used in Lemma~\ref{lemma:invar} after the following observation:
\begin{align*}
a_{m,n,k'} &= \sum\limits_{\substack{i_1+i_2+\cdots+i_n=(n-1)\frac{q^m-q}{q-1}+k' \\ 0\leq i_1,i_2,\cdots,i_n \leq \frac{q^m-q}{q-1} }}\bigg(\prod\limits_{j=1}^n x_j^{i_j(q-1)}\bigg)
\\&= \sum\limits_{i_n=0}^{\frac{q^m-q}{q-1}}\Bigg(\sum\limits_{\substack{i_1+i_2+\cdots+i_{n-1}=(n-1)\frac{q^m-q}{q-1}+k'-i_n \\ 0\leq i_1,i_2,\cdots,i_{n-1} \leq \frac{q^m-q}{q-1} }}\bigg(\prod\limits_{j=1}^{n-1} x_j^{i_j(q-1)}\bigg)\Bigg)x_n^{i_n(q-1)}
\\&= \sum\limits_{i_n=0}^{\frac{q^m-q}{q-1}} a_{m,n-1,q+k'-i_n}x_n^{i_n(q-1)}
\\&= \sum\limits_{i=k'}^{\frac{q^m-q}{q-1}} a_{m,n-1,q+k'-i}x_n^{i(q-1)}.
\end{align*}

Hence, the proven proposition and corollary provide us a set of invariant polynomials in the ring $\mathbb{F}_q[x_1,x_2,\cdots,x_n]/(x_1^{q^m},x_2^{q^m},\cdots,x_n^{q^m})$ which correspond to the the case `$k=1$' as for given values of $m$ and $n$, we have homogeneous invariant polynomials $a_{m,n,k'}$ for $0 \leq k' \leq \frac{q^m-q}{q-1}$, where the degree of $a_{m,n,k'}$ is equal to $(n-1)(q^m-q) + k'(q-1)$. The Hilbert series for the polynomials is thus given as $t^{(n-1)(q^m-q)} + t^{(n-1)(q^m-q)+q-1} + \cdots + t^{n(q^m-q)} = t^{(n-1)(q^m-q)}\frac{1-t^{q^m-1}}{1-t^{q-1}} = t^{(n-1)(q-1)}{m \brack 1}$, which is exactly the conjectured value. 

\section{Action of Steenrod operators} \label{sec:SO}

Now that we are done with the analysis of the `$k=1$' case of the Hilbert series, before moving on to further cases, we study the dependencies of the invariants we have talked about, viewed from the perspective of Steenrod operations. Throughout this section, we'll assume that $n=2$.

Recall (see \cite{LS}) that for a polynomial $f(x_1,x_2,\cdots,x_n) \in \mathbb{F}_q[x_1,x_2,\cdots,x_n]$, we define the action of Steenrod operators $P^i$ on $f$ as follows:
\begin{align*}
P(\xi)(f) :&= f(x_1 + x_1^q\xi, x_2 + x_2^q\xi, \cdots, x_n+x_n^q\xi)
\\&=\sum \limits_{i=0}^{\infty} P^i(f) \xi^i.
\end{align*}

It is easy to see that for any $f$, $P^0(f)=f$. Before moving towards the results of this section, we prove a lemma that motivates the study of these operators in the context of our problem.

\begin{lemma}
The Steenrod operators preserve the ideal $\mathfrak{m}^{[q^m]} = (x_1^{q^m},x_2^{q^m}, \cdots, x_n^{q^m})$, and hence, have a well-defined action on $Q$.
\end{lemma}

\begin{proof}
For any polynomial $f$ and $g \in \mathfrak{m}^{[q^m]}$, by the definition of the operators, we have
$$P^d(fg) =\sum\limits_{i=0}^d P^i(f)P^{d-i}(g)$$
for any $d$. Thus, we only need to verify our claim for the generators of the ring: $x_i^{q^m}$ for $1 \leq i \leq n$. This is easy to check as $P(\xi)(x_i^{q^m}) = (x_i+x_i^q\xi)^{q^m} = x_i^{q^m} + x_i^{q^{m+1}}\xi^{q^m}$ and $x_i^{q^m},x_i^{q^{m+1}} \in \mathfrak{m}^{[q^m]}$.
\end{proof}

\begin{remark}
The action of the Steenrod operators on the Dickson invariants $D_{n,i}$ can be found (when $q=p$) in \cite[\S II]{W}.
\end{remark}

In this section, we try to find a minimal additive basis for $Q^G$ with respect to the action of the Steenrod operators (For a connection to the 'hit problem' for Steenrod algebra, see \cite{RMW})

\subsection{$m=2$}
When $m=2$, we had found the invariants in Section~\ref{subsec:n=2} and they were named $y_{k'}$ where $0 \leq k' \leq q$. Using the same notation, and taking $k=k'+q$, we check the action of the first Steenrod operation on each $y_{k'}$. For compactness, we write $y_{k'} = \sum \limits_{i=k'}^q x_1^{(k-i)(q-1)}x_2^{i(q-1)}$. Then, we have:
$$P(y_{k'})=\sum \limits_{i=k'}^q (x_1+x_1^q\xi)^{(k-i)(q-1)}(x_2+x_2^q\xi)^{i(q-1)}.$$

The action of the first Steenrod operator on $y_k$ equals the coefficient of $\xi$ in the above expression. Hence,
\begin{align*}
P^1(y_{k'})&=\sum\limits_{i=k'}^q \{x_1^{(k-i)(q-1)} x_2^{i(q-1)-1}x_2^q i(q-1) + x_1^{(k-i)(q-1)-1}x_1^q (k-i)(q-1) x_2^{i(q-1)}\}\\
&=\sum\limits_{i=k'}^q\{-i.x_1^{(k-i)(q-1)}x_2^{(i+1)(q-1)} + (i-k).x_1^{(k-i+1)(q-1)}x_2^{i(q-1)}\}
\\&= - kx_1^{(k-k'+1)(q-1)}x_2^{k'(q-1)} -k\sum\limits_{i=k'+1}^q\{x_1^{(k-i+1)(q-1)}x_2^{i(q-1)}\} 
\\&\qquad+ \sum\limits_{i=k'}^q\{i(x_1^{(k-i+1)(q-1)}x_2^{i(q-1)} - x_1^{(k-i)(q-1)}x_2^{(i+1)(q-1)})\}
\\&= - kx_1^{(k-k'+1)(q-1)}x_2^{k'(q-1)} - ky_{k'+1} + k'x_1^{(k-k'+1)(q-1)}x_2^{k'(q-1)} + y_{k'+1} 
\\&= (1-k)y_{k'+1}.
\end{align*}

Thus, we have proved the following proposition:

\begin{proposition}
For all $k'$ such that $0 \leq k' <q$, $P^1(y_{k'})=(1-k)y_{k'+1}$, and hence, when $(1-k)\not\equiv 0 \pmod{p}$, we can use $P^1$ to obtain $y_{k'+1}$ from $y_{k'}$. 
\end{proposition}

Next, we prove a proposition that essentially will demonstrate that the action of the Steenrod operators on just 2 of the invariant polynomials is sufficient for generating all the other invariants.

\begin{proposition}
By applying the Steenrod operations $P_1, P_2, \cdots, P_{q-2}$ on $y_2$, we get non-zero scalar multiples of $y_3, y_4,\cdots, y_q$ respectively.
\end{proposition}

Using the same notation as defined earlier:
\begin{align*}
P(\xi)(y_2)&=\sum \limits_{i=2}^q \{P(\xi)(x_1)\}^{(q+2-i)(q-1)}\{P(\xi)(x_2)\}^{i(q-1)}
\\&=\sum \limits_{i=0}^{q-2} (x_1+x_1^q\xi)^{(q-i)(q-1)}(x_2+x_2^q\xi)^{(i+2)(q-1)}.
\end{align*}

For $1 \leq r \leq q-2, P^r(y_2)$ is the coefficient of $\xi^r$ in $P(y_2)$ which is
\begin{equation}
\sum\limits_{i=0}^{q-2} \sum\limits_{j=0}^r {(q-i)(q-1) \choose j}x_1^{(q-i+j)(q-1)} {(i+2)(q-1) \choose r-j} x_2^{(i+2+r-j)(q-1)}.
\end{equation}

Then, the following lemma proves the above proposition.

\begin{lemma}
$P^r(y_2) = {q-1 \choose q-r-1}y_{r+2}$ or
$$\sum\limits_{i=0}^{q-2} \sum\limits_{j=0}^r {(q-i)(q-1) \choose j}x_1^{(q-i+j)(q-1)} {(i+2)(q-1) \choose r-j} x_2^{(i+2+r-j)(q-1)}$$
$$={q-1 \choose q-r-1}\sum \limits_{i=0}^{q-r-2} x_1^{(q-i)(q-1)}x_2^{(i+r+2)(q-1)}.$$
\end{lemma}

\begin{proof}
In order to prove this identity, we first transform it to a simpler form. Firstly, cancel $x_1^{q(q-1)}x_2^{(r+2)(q-1)}$ from both sides. Next, replace $x_1^{q-1}$ and $x_2^{q-1}$ by $z_1$ and $z_2$ respectively, and next, replace $z_2/z_1$ by $z$. Thus, the identity to be proven transforms into:
\begin{equation}
\sum\limits_{i=0}^{q-2}\sum\limits_{j=0}^r {(q-i)(q-1) \choose j}{(i+2)(q-1) \choose r-j} z^{i-j} = {q-1 \choose q-r-1}\sum \limits_{i=0}^{q-r-2} z^{i}.
\end{equation}

Now, as we are in $\mathbb{F}_q$, we have ${(q-i)(q-1) \choose j} = {i \choose j}$ and ${(i+2)(q-1) \choose r-j} = {q-i-2 \choose r-j}$ by Lucas's Theorem. Finally, putting $k=i-j$ and suitably modifying the limits for summation, we get:
\begin{equation}
\sum\limits_{i=0}^{q-2}\sum\limits_{k=i-r}^i {i \choose k}{q-i-2 \choose q-2-r-k} z^k = {q-1 \choose q-r-1}\sum \limits_{k=0}^{q-r-2} z^k.
\end{equation}

Next, we notice that for $k>i$ and for $k<i-r$, the summand on the LHS becomes 0. So, we can sum over $k$ from $0$ to $q-2$. Then, the coefficient of $z^k$ on the LHS becomes equal to $\sum\limits_{i=0}^{q-2} {i \choose k}{q-i-2 \choose q-2-r-k}$. (For $k>q-r-2$, this is zero exactly as in the RHS). This is the coefficient of $a^kb^{q-2-r-k}$ in:
\begin{align*}
\sum\limits_{i=0}^{q-2} (1+a)^i(1+b)^{q-i-2} &= (1+a)^{q-2} \sum\limits_{i=0}^{q-2} \bigg(\frac{1+b}{1+a} \bigg)^{q-i-2}
\\&=(1+a)^{q-2} \frac{1-\bigg(\frac{1+b}{1+a}\bigg)^{q-1}}{1-\frac{1+b}{1+a}}
\\&=\frac{(1+a)^{q-1} - (1+b)^{q-1}}{a-b}
\\&=\frac{\sum\limits_{i=0}^{q-1}{q-1 \choose i}(a^i-b^i)}{a-b}
\\&=\sum\limits_{i=0}^{q-1} {q-1 \choose i} \sum\limits_{j=0}^{i-1}a^jb^{i-j-1}.
\end{align*}

The coefficient of $a^kb^{q-2-r-k}$ in the above expression is exactly ${q-1 \choose q-r-1}$ which is exactly the coefficient of $z^k$ in the RHS of the identity we are trying to prove. Hence the two expressions are equal in $\mathbb{F}_q$. 
\end{proof}

An important observation here is that ${q-1 \choose q-r-1} = {q-1 \choose r} \neq 0$ for the given values of $r$ (again by Lucas' Theorem). This signifies that given $y_0$ and $y_2$, we can generate all the other $y_k$'s (up to scalar multiples) by the action of the Steenrod operations on these.

\subsection{$m \geq 2$}
Now, we look at the case when $n$ is still 2 and $m$ is arbitrary. Then, corresponding to the `$k=1$' case, from Section~\ref{sec:k=1}, we have the invariants $a_{m,2,k'}$ where $0 \leq k' \leq \frac{q^m-q}{q-1}$. Then, we make the following conjecture about the relationships between these invariants by means of the Steenrod operations.

\begin{conjecture}
For a fixed $m$, consider the set $A:= \{ a_{m,2,k'} : 0\leq k' \leq \frac{q^m-q}{q-1} \}$ and the set $B:= \{a_{m,2,k'} : k'=0, 1+\frac{q^t-1}{q-1} \text{for } 1\leq t\leq m-1\}$. Then, it is possible to generate all the elements of $A$ by the action of Steenrod operations on the elements of the set $B$. $a_{m,2,0}$ generates itself and $a_{m,2,1}$. For any $t$ such that $1 \leq t < m-1$ and $k'=1+\frac{q^t-1}{q-1}$, $a_{m,2,k'}$ generates $a_{m,2,l}$ for $l=k', k'+1, k'+2, \cdots, k'+q^t$. For $t=m-1$, $a_{m,2,k'}$ generates $a_{m,2,l}$ for $l=k', k'+1,\cdots, k'+q^t-1$.
\end{conjecture}

What we have already proven agrees with this conjecture for the case $m=2$. The proof in the general case should involve the verification of identities involving binomial coefficients just like the case $m=2$. If we proceed in the same manner as above, it all boils down to proving that the expression
$$\sum\limits_{i=0}^{\frac{q^m - q^t - 2q + 2}{q-1}}{(q-1)(\frac{q^m-q}{q-1}-i) \choose j}{(q-1)(\frac{q^t+q-2}{q-1}+i) \choose r-j} \pmod{p}$$
is independent of $j$ when $0 \leq j \leq r$.

\section{The case $n=k=m-1$} \label{sec:n=k=m-1}

Recall that the conjectured Hilbert series is given as:
$$Hilb(Q^G,t)=\sum\limits_{k=0}^{\min(n,m)}t^{(n-k)(q^m-q^k)}{m \brack k}_{q,t}.$$
In this section, we'll investigate the case $n=k=m-1=2$ first, and then try to tackle the problem for arbitrary $n$. The reason for treating different values of $k$ separately stems from our belief that the invariants can be systematically categorised on the basis of the value of $k$ to which they correspond. For $n=2$ and $m=3$, the conjectured series is:
$$t^{2(q^3-1)} + t^{(q^3-q)}\frac{1-t^{q^3-1}}{1-t^{q-1}} + \frac{(1-t^{q^3-1})(1-t^{q^3-q})}{(1-t^{q^2-1})(1-t^{q^2-q})}.$$

The first two terms in the above sum are easy to account for: the first one corresponds to the invariant polynomial $x_1^{q^3-1}x_2^{q^3-1}$ whereas the second one was dealt with in Section~\ref{sec:k=1}.

Let the third term be referred to as $f(t)$ which is a polynomial of degree $2(q^3-q^2)$. An important observation is that $f(1/t) = t^{-2(q^3-q^2)}f(t)$. This implies that the coefficient of $t^a$ equals the coefficient of $t^{2(q^3-q^2)-a}$ in $f(t)$ where $0 \leq a \leq 2(q^3-q^2)$. Now, we try to see how the coefficients of $f(t)$ look like. Now,
\begin{align*}
f(t) &= \frac{(1-t^{q^3-1})(1-t^{q^3-q})}{(1-t^{q^2-1})(1-t^{q^2-q})}
\\&= (1-t^{q^3-1})(1-t^{q^3-q})(1+t^{q^2-1} + t^{2(q^2-1)} + \cdots)(1+t^{q^2-q}+t^{2(q^2-q)}+\cdots).
\end{align*}

\begin{lemma}
For any $a$ between $0$ and $q^3-q^2$, the coefficient of $t^a$ in $f(t)$ is exactly $1$ if $a$ is a non-negative integer linear combination of $q^2-1$ and $q^2-q$, and $0$ otherwise.
\end{lemma}

\begin{proof}
For $0 \leq a \leq q^3-q^2$, let us look at the coefficient of $t^a$ in the above product. Now, the first $2$ terms of the product will have to contribute $1$ (as $q^3-1$ and $q^3-q$ are greater than $a$). Hence, $a$ must be of the form $m(q^2-1) + n(q^2-q)$ for some $m,n \in \mathbb{N}_0$. In fact, for a given $a$, such a representation is unique. This is because if we have $2$ pairs $(m_1,n_1)$ and $(m_2,n_2)$ such that
$$a=m_1(q^2-1) + n_1(q^2-q)=m_2(q^2-1) + n_2(q^2-q),$$
then,
$$(m_1-m_2)(q+1)=(n_2-n_1)q.$$

And so $q|(m_1-m_2)$. If $m_1 \neq m_2$, one of them will have to be $\geq q$ and that would imply $a \geq q(q^2-1)$ which as a contradiction as $a \leq q^3-q^2$. Hence, $m_1=m_2$ and $n_1=n_2$.
\end{proof}

For $q^3-q^2 \leq a \leq 2(q^3-q^2)$, the coefficient of $t^a$ equals the coefficient of $t^{2(q^3-q^2)-a}$, and as, $0 \leq 2(q^3-q^2)-a \leq q^3-q^2$, we have already determined this coefficient. Hence, all the coefficients in $f(t)$ are either $0$ or $1$. Now, recall that $S=\mathbb{F}_q[x_1,x_2]$ and $Q=S/(x_1^{q^3},x_2^{q^3})$.

\begin{proposition} \label{prop:n=k=m-1}
Corresponding to the term $f(t)$ in the Hilbert series, we can find invariant polynomials in $Q$ which are images of Dickson invariant polynomials in $S$.
\end{proposition}

\begin{proof}
The invariant ring $S^G$ is generated by the polynomials $D_{2,0}$ and $D_{2,1}$ having degrees $q^2-1$ and $q^2-q$ respectively. Let their images in $Q$ be denoted by $S_0$ and $S_1$ respectively.

Now, by manual computation, one sees that $f(t) = 1 + t^{q^2-q} + t^{q^2-1} +$ higher degree terms. As degree($S_0$) = $q^2-q$ and degree($S_1$) = $q^2-1$, these $2$ polynomials correspond to the $2^{nd}$ and $3^{rd}$ monomials of $f(t)$ respectively. For any other $a$ such that $0 \leq a \leq q^3-q^2$, the coefficient of $t^a$ in $f(t)$ is $1$ iff $a=m(q^2-1)+n(q^2-q)$ for some $m,n\in\mathbb{N}_0$. Then, for such an $a$, the polynomial $S_1^mS_0^n$ is an invariant polynomial in $Q$ of degree $a$, and as $a < q^3-1$, we can be sure that $S_1^mS_0^n \neq 0$. An easy observation that will be used later is that, if $0 \leq a < q^3-q^2$, then $m \leq q$ and $n \leq q-1$. Next, we talk about $a$ when $q^3-q^2 < a \leq 2(q^3-q^2)$. 

As will be proved in Lemma~\ref{lem:nzero}, $S_1^qS_0^{q-1}\neq 0$. Then, for any $a$ in the above range, if the coefficient of $t^a$ is non-zero in $f(t)$, then $2(q^3-q^2)-a  = m(q^2-1)+n(q^2-q)$ for some $m,n \in \mathbb{N}_0$. Then, $S_1^{q-m}S_0^{q-1-n}$ is an invariant polynomial of degree $(q-m)(q^2-1) + (q-n-1)(q^2-q) = a$. It is well defined as $m\leq q$ and $n \leq q-1$. Also, $S_1^{q-m}S_0^{q-1-n} \neq 0$ as $S_1^{q-m}S_0^{q-1-n} \times S_1^mS_0^n = S_1^qS_0^{q-1} \neq 0$.
\end{proof}

\begin{lemma} \label{lem:nzero}
Using the same notation as in the previous proposition, $S_1^qS_0^{q-1} \neq 0$. 
\end{lemma}

\begin{proof}
For this proof, we need the explicit form of $S_0$ and $S_1$. $S_1$ is easy to determine from the definition of $D_{2,1}$ whereas for $S_0$, we use Remark~\ref{rem:DI}. So, we have:
\begin{align*}
S_0 &= x_1^{q^2-q} + x_1^{(q-1)^2}x_2^{q-1} + \cdots +x_2^{q^2-q}
\\S_1 &= x_1^{q^2-q}x_2^{q-1} + x_1^{(q-1)^2}x_2^{2(q-1)} + \cdots + x_1^{q-1}x_2^{q^2-q}.
\end{align*}

Easy to see that $S_0 = \frac{x_1^{q^2-1} - x_2^{q^2-1}}{x_1^{q-1}-x_2^{q-1}}$ and $S_1 = S_0x_2^{q-1} - x_2^{q^2-1}$. Then, $S_1^qS_0^{q-1} = S_0^{2q-1}x_2^{q^2-q} - S_0^{q-1}x_2^{q^3-q}$. Then, the coefficient of $x_1^{q^3-q}x_2^{q^3-2q^2+q}$ in $S_1^qS_0^{q-1}$
\\= the coefficient of $x_1^{q^3-q}x_2^{q^3-2q^2+q}$ in $S_0^{2q-1}x_2^{q^2-q}$
\\= the coefficient of $x_1^{q^3-q}x_2^{q^3-3q^2+2q}$ in $S_0^{2q-1}$
\\= the coefficient of $x_1^{q^3-q}x_2^{q^3-3q^2+2q}$ in $\Big(\frac{x_1^{q^2-1} - x_2^{q^2-1}}{x_1^{q-1}-x_2^{q-1}}\Big)^{2q-1}$
\\= the coefficient of $t^{q^3-q}$ in $\Big(\frac{1-t^{q^2-1}}{1-t^{q-1}}\Big)^{2q-1}$ (by substituting $x_1=tx_2$)
\\= the coefficient of $t^{q^2+q}$ in $\Big(\frac{1-t^{q+1}}{1-t}\Big)^{2q-1}$
\\= $-1.$

Thus, as the coefficient is non-zero, $S_1^qS_0^{q-1} \neq 0$.
\end{proof}

We also believe that none of the invariant polynomials talked about in the above proposition are scalar multiples of the invariant polynomials corresponding to the `$k=1$' term of the Hilbert series. If that is true, then we have computed invariants corresponding to every term of the Hilbert series and would not expect any other invariant polynomials to exist. 

\begin{conjecture}
None of the polynomials described by us in Proposition~\ref{prop:n=k=m-1} are scalar multiples of the invariant polynomials corresponding to the `$k=1$' case.
\end{conjecture}

Now, we make a brief comment in the case when $n=k=m-1$ for an arbitrary $n$. Then, the corresponding term of the Hilbert series is:
$${n+1 \brack n}_{q,t} = \prod\limits_{i=0}^{n-1} \frac{1-t^{q^{n+1}-q^i}}{1-t^{q^{n}-q^i}}.$$

In the denominator, we see terms having degrees $q^n-q^i$ for $0 \leq i \leq n-1$. This is exactly the set of degrees of the Dickson invariant polynomials $D_{n,0}, D_{n,1}, \cdots, D_{n,n-1}$ in $\mathbb{F}_q[x_1,x_2,\cdots,x_n]$. On computation, one sees that ${n+1 \brack n}$ is again a polynomial all of whose coefficients are 0 or 1. Thus, inspired by our work in the case $n=2$, we make the following conjecture.

\begin{conjecture} \label{conj:n=k=m-1}
For arbitrary $n$ and $m=n+1$, if we consider the $k=n$ term of the conjectured Hilbert series of $Q^G$, then one can find invariant polynomials corresponding to this term of the series that are images of invariant polynomials in $S$.
\end{conjecture}

\begin{remark}
The ideas developed above for the case $n=k=m-1=2$ can be used to work further on the case $n=2$ for an arbitrary $m$. In such a case, the $k=2$ term of the Hilbert series is ${m \brack 2}_{q,t} = \frac{(1-t^{q^m-1})(1-t^{q^m-q})}{(1-t^{q^2-1})(1-t^{q^2-q})}$. Once again, we can talk about non-negative integer linear combinations of $q^2-1$ and $q^2-q$. However, complications start to arise because the coefficients of $t^a$ in this series are not necessarily just $0$ or $1$.
\end{remark}

\section{The case $q=2$} \label{sec:q=2}

Now, we record some results about the structure of the invariant ring when $q=2$. We assume that $n=2$ and consider the conjectured Hilbert series given as (taking $m\geq 3)$:
$$t^{2(2^m-1)} + t^{2^m-2}\frac{1-t^{2^m-1}}{1-t} + \frac{(1-t^{2^m-1})(1-t^{2^m-2})}{(1-t^{3})(1-t^2)}.$$

The $3^{rd}$ term in the above sum has been discussed comprehensively in the previous section. Another conjecture that could be coupled with Conjecture~\ref{conj:n=k=m-1} is that this term corresponds to the images of invariant polynomials in $S$. Throughout this section, we refer to polynomials of the ring $S^G$ as Dickson invariant polynomials.

\begin{conjecture} \label{conj:3rd term}
For $q=n=2$ and arbitrary $m$, the `$k=3$' term of the conjectured Hilbert series corresponds to invariant polynomials of $Q$ that are images of Dickson invariant polynomials in $S$.
\end{conjecture}

However, in this section, we are more concerned about the $2^{nd}$ term of the Hilbert series. The $2^{nd}$ term is equal to $t^{2^m-2} + t^{2^m-1} + t^{2^m} + \cdots + t^{2^{m+1}-4}$. From Proposition~\ref{prop:k=1}, we have that the invariant polynomials corresponding to this are given by $y_{k'}$ where $0 \leq k' \leq 2^m-2$ and:
$$y_{k'} = x_1^{2^m-2}x_2^{k'} + x_1^{2^m-3}x_2^{k'+1}+\cdots+x_1^{k'+1}x_2^{2^m-3}+x_1^{k'}x_2^{2^m-2}.$$

The degrees of $y_0$ and $y_1$ are $2^m-2$ and $2^m-1$ respectively. As these are less than $2^m$, the polynomials are actually invariant in $S$ itself. Hence, they are Dickson polynomials. We claim that each $y_k'$ is the image of a Dickson polynomial. To show this, we first note that $D_{2,0}=x_1^2+x_1x_2+x_2^2$. Thus,
\begin{align*}
D_{2,0} \times y_{k'} &= (x_1^2+x_1x_2+x_2^2) \sum\limits_{i=k'}^{2^m-2} x_1^{2^m-2+k'-i}x_2^{i}
\\&=\sum\limits_{i=k'}^{2^m-2} x_1^{2^m+k'-i}x_2^{i}+\sum\limits_{i=k'}^{2^m-2} x_1^{2^m-1+k'-i}x_2^{i+1}+\sum\limits_{i=k'}^{2^m-2} x_1^{2^m-2+k'-i}x_2^{i+2}
\\&=\sum\limits_{i=k'+2}^{2^m-2} x_1^{2^m+k'-i}x_2^{i}+\sum\limits_{i=k'+1}^{2^m-3} x_1^{2^m-1+k'-i}x_2^{i+1}+\sum\limits_{i=k'}^{2^m-4} x_1^{2^m-2+k'-i}x_2^{i+2}
\\&\qquad+x_1^{2^m-1}x_2^{k'+1} + x_1^{2^m-1}x_2^{k'+1} + x_1^{k'+1}x_2^{2^m-1} + x_1^{k'+1}x_2^{2^m-1}
\\&=\sum\limits_{i=k'+1}^{2^m-3} x_1^{2^m+k-1'-i}x_2^{i+1}+\sum\limits_{i=k'+1}^{2^m-3} x_1^{2^m-1+k'-i}x_2^{i+1}+\sum\limits_{i=k'+2}^{2^m-2} x_1^{2^m+k'-i}x_2^{i}
\\&=\sum\limits_{i=k'+2}^{2^m-2} x_1^{2^m+k'-i}x_2^{i}
\\&=y_{k'+2}.
\end{align*}

Thus, $D_{2,0}y_{k'}=y_{k'+2}$, and so, as $y_0$ and $y_1$ are images of Dickson invariant polynomials, we can conclude the same for $y_{k'}$ for all $k'$.

Also, the polynomial $x_1^{2^{m+1}-2} + x_1^{2^{m+1}-3}x_2 + \cdots + x_1x_2^{2^{m+1}-3} + x_2^{2^{m+1}-2}$ is an invariant polynomial in $S^G$ (as a corollary of Remark~\ref{rem:DI}). Its image in $Q$ is $x_1^{2^m-1}x_2^{2^m-1}$ which is the invariant polynomial corresponding to the first term of the Hilbert series. Hence, that too is an image of a Dickson invariant polynomial. We conclude with the following theorem:

\begin{theorem}
Assuming Conjecture~\ref{conj:3rd term}, all the polynomials in the invariant ring $Q^G$ are images of Dickson invariant polynomials, i.e. images of polynomials in $S^G$, when the underlying field is $\mathbb{F}_2$ and $n=2$.
\end{theorem}

\section{The parabolic conjecture: $m=2$} \label{sec:para}

In this section, we discuss a generalisation of Conjecture~\ref{conj:prob} to parabolic subgroups (see \cite{LRS}). A parabolic subgroup $P_{\alpha}$ of $G$ specified by a composition $\alpha=(\alpha_1,\alpha_2,\cdots,\alpha_l)$ of $n$, so that $|\alpha| := \alpha_1+\alpha_2+\cdots+\alpha_l=n$, is defined to be the subgroup consisting of block upper-triangular matrices of the form:
\[
  \begin{bmatrix}
    g_1 & * &\cdots& * \\
    0 & g_2 &\cdots& * \\
    . & .   &.     & . \\
    . & .   &.     & . \\
    . & .   &.     & . \\
    0 & 0   &\cdots& g_l\\
  \end{bmatrix}
\]
where each diagonal block $g_i$ is an $\alpha_i\times \alpha_i$ matrix. For such subgroups of $G$, we wish to study the structure of $Q^{P_{\alpha}}$ where $Q=\mathbb{F}_q[x_1,x_2,\cdots,x_n]/(x_1^{q^m},x_2^{q^m},\cdots,x_n^{q^m})$. Denoting partial sums of the components of $\alpha$ by $A_i:=\alpha_1+\alpha_2+\cdots+\alpha_i$, we define:
$${n \brack \alpha} := \frac{\prod\limits_{j=1}^{n-1}(1-t^{q^n-q^j})}{\prod\limits_{i=1}^l\prod\limits_{j=0}^{\alpha_i-1}(1-t^{q^{A_i} - q^{A_{i-1}+j}})}.$$
For two compositions $\alpha$ and $\beta$ both of length $l$, we define a partial order such that $\beta \leq \alpha$ iff $\beta_i\leq\alpha_i$ for all $i$. With this definition, we are ready to state the conjecture for the Hilbert series for $Q^{P_{\alpha}}$ as mentioned in \cite{LRS}. Define $B_i=\beta_1+\beta_2+\cdots+\beta_i$.

\begin{conjecture} \label{conj:para}
Fixing $m$, $n$ and a composition $\alpha$ of $n$, we have
$$Hilb(Q^{P_{\alpha}},t) = \sum\limits_{\substack{\beta:\beta\leq\alpha \\ |\beta|\leq m}} t^{e(m,\alpha,\beta)}{m \brack \beta, m-|\beta|}$$
where, $e(m,\alpha,\beta):=\sum\limits_{i=1}^l (\alpha_i-\beta_i)(q^m-q^{B_i}).$
\end{conjecture}

The $n=1$ case of Conjectures~\ref{conj:prob} and \ref{conj:para} coincide, and so, we assume $n\geq 2$ here onwards. The $m=1$ case of this conjecture has already been dealt with in \cite{LRS}. In this section, we analyse the $m=2$ case. 

Our aim is to show that for arbitrary $n$ and $\alpha$, we can find invariant polynomials corresponding to each term in the conjectured Hilbert series. In the summation over $\beta$ in the conjecture, the condition that $|\beta|\leq m=2$ implies that we have 4 cases:
\begin{itemize}
\item $\beta=(0,0,\cdots,0)$
\item $\beta=e_r=(0,0,\cdots, 0 ,1, 0, \cdots, 0)$ where $1$ is in the $r$th position, $1 \leq r \leq l$
\item $\beta=e_r+e_s$ where $1 \leq r, s \leq l$ and $r \neq s$
\item $\beta=2e_r$ where $1\leq r \leq l$ (provided $\alpha_r\geq 2$ as $\beta \leq \alpha$)
\end{itemize}

We'll deal with each case separately and show that invariant polynomials can be found corresponding to each term in the summation. The idea will be to fiddle with the polynomials already constructed in Section~\ref{sec:m=2}. Before we start, we state an easily proved counterpart of Lemma~\ref{lemma:intro1} for parabolic subgroups.
\begin{lemma}
A polynomial is invariant under the action of $P_{\alpha}$ if it is invariant under diagonal matrices, permutations of variables within a block, and the substitutions $x_i \to x_i + x_j$ whenever $1 \leq j < i \leq n$.
\end{lemma}
Here, by a permutation within a block, we mean any permutation of the variables $x_{A_{i-1}+1},x_{A_{i-1}+2},\dots,x_{A_i}$ for a fixed $i$ between $1$ and $l$.
\\
\\\textbf{Case 1:} Suppose $\beta=(0,0,\cdots,0)$. Then, $e(m,\alpha,\beta)=|\alpha|(q^m-1)=n(q^2-1)$, and ${2 \brack \beta, 2} = 1$. So, $t^{e(m,\alpha,\beta)}{m \brack \beta, m-|\beta|}=t^{n(q^2-1)}$. Consider the polynomial
$$a_n:=\prod\limits_{j=1}^n x_j^{q^2-1}.$$
We already know that $a_n$ is invariant under the action of $G$. In particular, it is invariant under the action of $P_{\alpha}$. As its degree is $n(q^2-1)$, it corresponds to the $t^{n(q^2-1)}$ term of the Hilbert series.
\\
\\\textbf{Case 2:} Suppose $\beta=e_r$ for some $r$. Then, $e(m,\alpha,\beta) = A_{r-1}(q^2-1) + (n-A_{r-1} - 1)(q^2-q)$ and ${2 \brack \beta, 2} = \frac{1-t^{q^2-1}}{1-t^{q-1}} = 1+ t^{q-1} + t^{2(q-1)} + \cdots + t^{q^2-q}$. So, $t^{e(m,\alpha,\beta)}{m \brack \beta, m-|\beta|}=t^{A_{r-1}(q^2-1)}t^{(n-A_{r-1} - 1)(q^2-q)}(1+ t^{q-1} + t^{2(q-1)} + \cdots + t^{q^2-q})$. For each $r$, consider the polynomials
$$b_{r,k}:=\prod\limits_{j=1}^{A_{r-1}}x_j^{q^2-1} \times \sum\limits_{\substack{i_{A_{r-1}+1}+i_{A_{r-1}+2}+\cdots+i_n=(n-A_{r-1}-1)q+k \\ 0\leq i_{A_{r-1}+1},i_{A_{r-1}+2},\cdots,i_n \leq q }}\bigg(\prod\limits_{j=A_{r-1}+1}^n x_j^{i_j(q-1)}\bigg)$$
where $0 \leq k \leq q$. Let the 2 multiplicands in the above product be $F_{r,k}$ and $G_{r,k}$ respectively. Then, degree of $F_{r,k}$ is $A_{r-1}(q^2-1)$ and degree of $G_{r,k}$ is $\{(n-A_{r-1}-1)q+k\}(q-1)$. Thus, if it is invariant, the polynomial $b_{r,k}$ corresponds to the $t^{A_{r-1}(q^2-1)}t^{\{(n-A_{r-1} - 1)q + k\}(q-1)} = t^{A_{r-1}(q^2-1)}t^{(n-A_{r-1} - 1)(q^2-q)}t^{k(q-1)}$ term of the series. Hence, if we let $k$ vary from $0$ to $q$, we can account for all the terms in $t^{e(m,\alpha,\beta)}{m \brack \beta, m-|\beta|}$.
\\
\\Now, we check the invariance of $b_{r,k}$. As, all exponents are multiples of $q-1$, diagonal matrices leave $b_{r,k}$ invariant. Any permutation of the block $i$ when $1 \leq i \leq r-1$ affects only $F_{r,k}$ and leaves it invariant, and $r \leq i \leq n$ affects only $G_{r,k}$ and leaves it invariant (both due to symmetry). Now, we check the action of substitutions. For $x_i \rightarrow x_i + x_j$ where $j < i \leq r-1$, the substitution affects only $F_{r,k}$ and invariance is trivial as powers greater than or equal to $q^2$ become 0. Next, for $x_i \rightarrow x_i + x_j$ where $r \leq j < i$, the substitution affects only $G_{r,k}$ and we know from Section~\ref{subsec:n=arbit} that this substitution leaves it invariant. Finally, consider the substitution $x_i \rightarrow x_i+x_j$ where $j<r\leq i$. This again leaves $b_{r,k}$ invariant because the exponent of $x_j$ is already $q^2-1$ and any higher powers would become zero. Thus, we conclude that $b_{r,k}$ is an invariant polynomial.
\\
\\\textbf{Case 3:} Suppose $\beta=e_r+e_s$ where $1\leq r<s \leq l$. Then, $e(m,\alpha,\beta) = A_{r-1}(q^2-1) + (A_{s-1}-A_{r-1}-1)(q^2-q)$ and ${2 \brack \beta, 2} = \frac{1-t^{q^2-1}}{1-t^{q-1}} = 1+ t^{q-1} + t^{2(q-1)} + \cdots + t^{q^2-q}$. So, $t^{e(m,\alpha,\beta)}{m \brack \beta, m-|\beta|}=t^{A_{r-1}(q^2-1)}t^{(A_{s-1}-A_{r-1} - 1)(q^2-q)}(1+ t^{q-1} + t^{2(q-1)} + \cdots + t^{q^2-q})$. For each such $r$ and $s$, consider the polynomials
$$c_{r,s,k}:=\prod\limits_{j=1}^{A_{r-1}}x_j^{q^2-1} \times \sum\limits_{\substack{i_{A_{r-1}+1}+i_{A_{r-1}+2}+\cdots+i_{A_{s-1}}=(A_{s-1}-A_{r-1}-1)q+k \\ 0\leq i_{A_{r-1}+1},i_{A_{r-1}+2},\cdots,i_{A_{s-1}} \leq q }}\bigg(\prod\limits_{j=A_{r-1}+1}^{A_{s-1}} x_j^{i_j(q-1)}\bigg)$$
where $0 \leq k \leq q$. Let the 2 multiplicands in the above product be $S_{r,s,k}$ and $T_{r,s,k}$ respectively. Then, degree of $S_{r,s,k}$ is $A_{r-1}(q^2-1)$ and degree of $T_{r,s,k}$ is $\{(A_{s-1}-A_{r-1}-1)q+k\}(q-1)$. Thus, if it is invariant, the polynomial $c_{r,s,k}$ corresponds to the $t^{A_{r-1}(q^2-1)}t^{\{(A_{s-1}-A_{r-1} - 1)q + k\}(q-1)} = t^{A_{r-1}(q^2-1)}t^{(A_{s-1}-A_{r-1} - 1)(q^2-q)}t^{k(q-1)}$ term of the series. Hence, if we let $k$ vary from $0$ to $q$, we can account for all the terms in $t^{e(m,\alpha,\beta)}{m \brack \beta, m-|\beta|}$. The proof that $c_{r,s,k}$ is invariant is exactly the same as in the previous case.
\\
\\\textbf{Case 4:} Suppose $\beta=2e_r$ for some $r$ such that $\alpha_r \geq 2$. Then, $e(m,\alpha,\beta) = A_{r-1}(q^2-1)$ and ${2 \brack \beta, 2} =1$. So, $t^{e(m,\alpha,\beta)}{m \brack \beta, m-|\beta|}=t^{A_{r-1}(q^2-1)}$. For each $r$, consider the polynomial
$$d_r:=\prod\limits_{j=1}^{A_{r-1}} x_j^{q^2-1}.$$
We claim that each $d_r$ is invariant under $P_{\alpha}$, and thus, corresponds to the $t^{A_{r-1}(q^2-1)}$ term of the Hilbert series. $d_r$ is trivially invariant under diagonal matrices. Also, as we are taking product for all variables corresponding to the first $r-1$ blocks, $d_r$ is invariant under the permutation of variables within a block. Finally, for any $j$, if we substitute $x_j \rightarrow x_j + x_k$ where $k<j$, we end up with $d_r$ again as exponents greater than or equal to $q^2$ become zero. Thus, our claim is true.
\\
\\The following example helps illustrate the above cases:

\begin{example}
Consider the case when $q=3$, $n=6$ and $m=2$ and the composition $\alpha=(2,1,3)$ of 6. In Case 1, we take $\beta=(0,0,0)$. The corresponding term of the Hilbert series is $t^{48}$ and the associated invariant polynomial is $a_6=x_1^8x_2^8x_3^8x_4^8x_5^8x_6^8$.

In Case $2$, $\beta=(1,0,0)$ corresponds to the term $t^{30}(1+t^2+t^4+t^6)$. The associated invariant polynomials are the polynomials $4$ polynomials that are invariant under the action of $Gl_6(\mathbb{F}_3)$ (i.e., the $y_{6,k}$'s for $0 \leq k \leq 3$ in the notation of Section~\ref{subsec:n=arbit}). Next, taking $\beta=(0,1,0)$ the corresponding term of the Hilbert series is $t^{16}t^{18}(1+t^2+t^4+t^6)$. The associated invariant polynomials are obtained by taking the product of the monomial $x_1^8x_2^8$ with polynomials in $x_3, x_4, x_5,$ and $x_6$ that are invariant under the action of $Gl_4(\mathbb{F}_3)$ (i.e., the $y_{4,k}$'s for $0 \leq k \leq 4$). Next, take $\beta=(0,0,1)$. The corresponding term in the series is $t^{24}t^{12}(1+t^2+t^4+t^6)$. The associated invariant polynomials will be products of the monomial $x_1^8x_2^8x_3^8$ with polynomials in $x_4, x_5,$ and $x_6$ that are invariant under the action of $Gl_3(\mathbb{F}_3)$ (i.e., the $y_{3,k}$'s for $0 \leq k \leq 4$). For example, 
\\$b_{3,1}=x_1^8x_2^8x_3^8 \Big(x_4^6x_5^6x_6^2 + (x_4^6x_5^4 + x_4^4x_5^6)x_6^4 + (x_4^6x_5^2+x_4^4x_5^4+x_4^2x_6^6)x_6^6 \Big)$.
\\
\\Now  we move on to Case 3. $\beta=(1,1,0)$ corresponds to the term $t^6(1+t^2+t^4+t^6)$. The associated invariant polynomials, the $c_{1,2,k}$'s are the polynomials in $x_1$ and $x_2$ that are invariant under the action of $Gl_2(\mathbb{F}_3)$ (i.e., the $y_{2,k}$'s for $0\leq k \leq 3$). Next, taking $\beta=(1,0,1)$, the corresponding term of the Hilbert series is $t^{24}(1+t^2+t^4+t^6)$ and the associated invariant polynomials are the polynomials in $x_1,x_2,x_3$ that are invariant under the action of $Gl_3(\mathbb{F}_3)$ (i.e., the $y_{3,k}$'s for $0\leq k \leq 3$). Finally, when $\beta=(0,1,1)$, the term of the series is $t^{16}(1+t^2+t^4+t^6).$ The associated invariant polynomials are $x_1^8x_2^8$, $x_1^8x_2^8x_3^2$, $x_1^8x_2^8x_3^4$ and $x_1^8x_2^8x_3^6$.
\\
\\Finally, we look at Case 4. When $\beta=(2,0,0)$, the corresponding term of the Hilbert series is 1 and the associated polynomial is $d_1=1$. For $\beta=(0,0,2)$, the corresponding term is $t^{24}$ and the associated invariant polynomial is $d_3=x_1^8x_2^8x_3^8$. This finishes our analysis of all the cases.
\end{example}

We conclude with the following conjecture:

\begin{conjecture}
Given $m=2$ and arbitrary $n$ and fixing a composition $\alpha$ of $n$, the polynomials $a_n$, $b_{r,k}$'s, $c_{r,s,k}$'s and $d_r$'s provide an additive basis for the invariant ring $Q^{P_{\alpha}}$, thus proving Conjecture~\ref{conj:para} when $m=2$.
\end{conjecture}

As we have already checked the invariance of the above polynomials, the proof of the above conjecture hinges on showing that, up to scalar multiples, these are the only homogeneous invariant polynomials in $Q$.

\begin{remark}
Among the polynomials defined above, the polynomials $a_n$ and $b_{1,k}$ for $0\leq k\leq q$ are the ones that are invariant under the action of the entire group $G$. In the notation of Section~\ref{subsec:n=arbit}, $a_n$ is equal to $z_n$ and $b_{1,k}$ is equal to $y_{n,k}$ for all $k$.
\end{remark}

\section{Acknowledgements}

I would like to thank the S.N. Bose Scholars Program, specifically Dr. Aseem Z. Ansari, Department of Biochemistry, University of Wisconsin-Madison, and the entire team of Winstep Forward, IUSSTF and SERB for having provided me with this peerless opportunity of getting a hands-on research experience and for facilitating my stay in the United States. I would also like to acknowledge the crucial role played by my mentor Dr. Steven V. Sam, Department of Mathematics, University of Wisconsin-Madison, who provided me a better insight into the field as well as the problem that I was working on, and elaborate discussions with whom helped me systematize my working style. Finally, I wish to thank my fellow Bose-Khorana scholars at UW Madison, spending time with whom was an immense pleasure and a source of great inspiration.


\begin{thebibliography}{6}

\bibitem{RMW} R.M.W. Wood, Problems in the Steenrod Algebra. {\it Bull. London Math. Soc.} (1998) \textbf{30} (5): 449-517.

\bibitem{LED} L.E. Dickson, A fundamental system of invariants of the general modular linear group with a solution of the form problem. {\it Trans. Amer. Math. Soc.} {\bf 12} (1911), 75–98.

\bibitem{LRS} J. Lewis, V. Reiner, D. Stanton, Invariants of $Gl_n(\mathbb{F}_q)$ in polynomials mod Frobenius powers. {\it Proc. Roy. Soc. Edinburgh A} to appear., {\tt  arXiv:1403.6521v3}

\bibitem{LS} Larry Smith, Polynomial invariants of finite groups: A survey of recent developments. {\it Bull. Amer. Math. Soc. (N.S.)} {\bf 34} (1997), no.~3, 211--250.

\bibitem{LT} Tianxin Cai, {\it The Book of Numbers}, World Scientific Publishing Co. Pte. Ltd.

\bibitem{W} Clarence Wilkerson, A primer on the Dickson invariants. {\it Contemporary Math.} {\bf 19} (1983), 421--434.

\end{thebibliography}
\end{document}